\documentclass[12pt,letterpaper,reqno]{amsart}

\usepackage{amsmath}
\usepackage[latin9]{inputenc}
\usepackage{color}
\usepackage{amsthm}
\usepackage{amstext}
\usepackage{amssymb}
\usepackage{latexsym}
\usepackage{amscd}
\usepackage{amsfonts}
\usepackage{euler}
\usepackage{setspace}
\usepackage{bbm}
\usepackage{enumerate}

\makeatletter

\numberwithin{equation}{section}
\numberwithin{figure}{section}

\newlength{\baseunit}

\newcount{\numlines}
\newtheorem{thm}{Theorem}

\newtheorem{lem}{Lemma}

\newtheorem{remark}{Remark}
\newtheorem{problem}{Problem}
\newcommand{\comments}[1]{}
\numberwithin{lem}{section}

\makeatother

\begin{document}
\title{The Brenner-Hochster-Koll\'ar and Whitney Problems for Vector-Valued Functions
and Jets}
\author{Charles L. Fefferman}
\thanks{The first author is partially supported by NSF and ONR grants DMS
09-01040 and N00014-08-1-0678.}
\author{Garving K. Luli}
\maketitle


\section{Introduction}

In \cite{Feff-Kollar}, the first author and J. Koll\'ar studied the
following problem

\begin{problem}[The Brenner-Hochster-Koll\'ar Problem]
\label{problem3-1} Let $X$ be a space of continuous functions on a topological space $E$. Suppose we are given $%
\mathbb{R}^{s}$-valued functions $f_{1},\ldots ,f_{d}$, and $\phi $ on $E$. How can we decide whether there exist $\phi _{1},\ldots
,\phi _{d}\in X$ such that
\begin{equation}
\sum_{i=1}^{d}\phi _{i}f_{i}=\phi \text{ on } E?  \label{kollar}
\end{equation}
\end{problem}

For $X=C^{0}(E)$ with $E=\mathbb{R}^n$, \cite{Feff-Kollar} gives two effective
methods (one analytic and the other algebraic) for solving this problem when
the given functions $f_{1},\ldots ,f_{d},\phi $ are polynomials. (\cite{Feff-Kollar} also treats more general $E$.) Problem \ref%
{problem3-1} in that case arose from algebraic geometry; see Brenner \cite{Brenner}, Epstein-Hochster \cite{Hochster}, and Koll\'ar \cite{kollar}.
When $\phi _{1},\ldots ,\phi _{d}$ exist, the algebraic method in \cite%
{Feff-Kollar} produces semi-algebraic $\phi _{1},\ldots ,\phi _{d}$. On the
other hand, the analytic method solves Problem \ref{problem3-1} for $%
X=C^{0}\left( \mathbb{R}^{n}\right) $ without assuming that $f_{1},\ldots
,f_{d},\phi $ are polynomials. Here, we extend the analytic method in \cite%
{Feff-Kollar} to solve Problem \ref{problem3-1} for $X=C^{m}(\mathbb{R}^{n})$
(space of real-valued functions whose derivatives up to order $m$ are
continuous and bounded on $\mathbb{R}^{n}$) and $C^{m,\omega }(\mathbb{R}%
^{n})$ (space of real-valued functions whose m-th derivatives have modulus
of continuity $\omega $; see Section \ref{notation} for more details). These
cases were left open in \cite{Feff-Kollar}.

Our work on Problem \ref{problem3-1} relates to

\begin{problem}[Whitney's Extension Problem]
\label{whitneyproblem}Let $X$ denote a function space. Suppose we are given
a compact set $E\subset \mathbb{R}^{n}$ and a function $f:E\rightarrow \mathbb{%
R}$. How can we decide whether there exists $F\in X$ such that $F=f$ on $E$?
\end{problem}

For $X=C^{m,\omega}(\mathbb{R}^{n})$ and $X=C^{m}(\mathbb{R}^{n})$, Problem %
\ref{whitneyproblem} was solved in \cite{F5,F1}, building on previous work
of H. Whitney \cite{Whitney1,Whitney2}, Brudnyi-Shvartsman \cite%
{brudny1,brudnyi1,brudnyi4,brudnyi2,brudnyi3}, and Bierstone-Milman-Paw\l %
ucki \cite{Bierstone-Milman}.

For $X=C^{m}(\mathbb{R}^{n})$, we will solve a more general problem that
includes both the Brenner-Hochster-Koll\'ar and the Whitney problems as special
cases. We believe that this general problem is of independent interest.

To facilitate the statement of our generalization of Problems \ref%
{problem3-1} and \ref{whitneyproblem}, we introduce a few definitions and a
bit of notation.

Let $C^{m}(\mathbb{R}^{n},\mathbb{R}^{d})$ be the space of $\mathbb{R}^{d}$%
-valued functions whose derivatives up to order $m$ are continuous and
bounded on $\mathbb{R}^{n}$. We write $\mathcal{P}_{m,n}$ to denote the
vector space of all (real) polynomials of degree at most $m$ on $\mathbb{R}%
^{n}$. For real-valued functions $F$, $J_{x}^{m}F$ stands for the m-jet at $%
x $, which we identify with the Taylor polynomial
\begin{equation*}
\hat{x}\longmapsto \sum_{\left\vert \alpha \right\vert \leq m}\frac{1}{%
\alpha !}\left( \partial ^{\alpha }F\right) \left( x\right) \left( \hat{x}%
-x\right) ^{\alpha }.
\end{equation*}%
Thus, the ring of $\mathcal{R}_{m,n}^{x}$ of $m$-jets of functions at $x$ is
identified with $\mathcal{P}_{m,n}$, the space of real m-th degree
polynomials on $\mathbb{R}^n$; and the multiplication $\odot _{m,n}^{x}$ in $%
\mathcal{R}_{m,n}^{x}$ may be regarded as a multiplication on $\mathcal{P}%
_{m,n}$. Here, the multiplication $\odot _{m,n}^{x}$ is defined as $P\odot
_{m,n}^{x}Q\equiv J_{x}^{m}\left( PQ\right) $ for $P,Q\in \mathcal{R}%
_{m,n}^{x}$. For vector-valued functions $\vec{F}=\left( F_{1},\ldots
,F_{d}\right) $, we write $J_{x}^{m}\vec{F}$ to denote the vector $%
(J_{x}^{m}F_{1},\ldots ,J_{x}^{m}F_{d})\in $ $\underset{d}{\underbrace{%
\mathcal{P}_{m,n}\oplus \ldots \oplus \mathcal{P}_{m,n}}}=(\mathcal{P}%
_{m,n})^{d}$. We regard $(\mathcal{P}_{m,n})^{d}$ as an $\mathcal{R}%
_{m,n}^{x}$-module by the multiplication rule $Q\odot _{m,n}^{x}\left(
P_{1},\ldots ,P_{d}\right) =\left( Q\odot _{m,n}^{x}P_{1},\ldots ,Q\odot
_{m,n}^{x}P_{d}\right) $.

To motivate the next few definitions, we note that the $m$-jet $\left(
P_{1},\ldots ,P_{d}\right) $ at $x$ of any solution of (\ref{kollar})
belongs to
\begin{equation}
\vec{H}\left( x\right) =\left\{ \vec{P}=\left( P_{1},\ldots ,P_{d}\right)
:P_{1}\left( x\right) f_{1}\left( x\right) +\ldots +P_{d}\left( x\right)
f_{d}\left( x\right) =\phi \left( x\right) \right\} \subset (\mathcal{P}%
_{m,n})^{d}\text{.}  \label{H(x)}
\end{equation}%
Here, the affine subspace $\vec{H}\left( x\right) \subset (\mathcal{P}%
_{m,n})^{d}$ may be computed from (\ref{kollar}) by elementary linear
algebra. Perhaps $\vec{H}\left( x\right) $ is empty. (By convention, we
allow the empty set, single points, and all of $(\mathcal{P}_{m,n})^{d}$ as
affine subspaces of $(\mathcal{P}_{m,n})^{d}$.) If $\vec{H}\left( x\right) $
is non-empty and if $\vec{P}_{x}^{0}$ is any element of $\vec{H}\left(
x\right) $, then we may express
\begin{equation*}
\vec{H}\left( x\right) =\vec{P}_{x}^{0}+\vec{I}\left( x\right) ,
\end{equation*}%
where $\vec{I}\left( x\right) =\left\{ \vec{P}=\left( P_{1},\ldots
,P_{d}\right) :P_{1}\left( x\right) f_{1}\left( x\right) +\ldots
+P_{d}\left( x\right) f_{d}\left( x\right) =0\right\} $ is an $\mathcal{R}%
_{m,n}^{x}$-submodule of the $\mathcal{R}_{m,n}^{x}$-module $(\mathcal{P}%
_{m,n})^{d}$.

A vector of functions $\vec{F}\in C^{m}\left( \mathbb{R}^{n},\mathbb{R}%
^{d}\right) $ solves equation (\ref{kollar}) if and only if $J_{x}^{m}\vec{F}%
\in \vec{H}\left( x\right) $ for all $x\in \mathbb{R}^{n}$.

Similarly, let $f:E\rightarrow \mathbb{R}$ be as in Problem \ref%
{whitneyproblem}. For $x \in E$ define
\begin{equation}
H(x)=\{P \in \mathcal{P}_{m,n} : P(x)=f(x)\},  \label{real-value1}
\end{equation}
\begin{equation}
f^x = \text {the constant polynomial } \hat{x} \longmapsto f(x),
\end{equation}
and
\begin{equation*}
I(x)=\{P \in \mathcal{P}_{m,n} : P(x)=0\}.
\end{equation*}
Then $H(x)=f^x+I(x)$, $I(x)$ is an ideal in $\mathcal{R}^x_{m,n}$, and a
function $F \in C^m(\mathbb{R}^n, \mathbb{R})$ satisfies $F=f$ on $E$ if and
only if $J_x^m F \in H(x)$ for all $x \in E$.

The above remarks motivate the following definitions. Fix integers $m\geq 0$%
, $n\geq 1$, $d\geq 1$. Let $E\subset \mathbb{R}^{n}$ be compact. A
\underline{bundle} over $E$ is a family $\mathcal{\vec{H}}=\left( \vec{H}%
\left( x\right) \right) _{x\in E}$ of (possibly empty) affine subspaces $%
\vec{H}\left( x\right) \subset (\mathcal{P}_{m,n})^{d}$, parametrized by the
points $x\in E$, such that each non-empty $\vec{H}\left( x\right) $ has the
form
\begin{equation*}
\vec{H}\left( x\right) =\vec{P}^{x}+\vec{I}\left( x\right)
\end{equation*}%
for some $\vec{P}^{x}\in (\mathcal{P}_{m,n})^{d}$ and some $\mathcal{R}%
_{m,n}^{x}$-submodule $\vec{I}\left( x\right) $ of $(\mathcal{P}_{m,n})^{d}$.

We make no assumptions as to how $\vec{H}\left( x\right) $, $\vec{P}^{x}$, $%
\vec{I}\left( x\right) $ depend on $x$.

We call $\vec{H}\left( x\right) $ the \underline{fiber} of $\mathcal{\vec{H}}
$ at $x$. If $\mathcal{\vec{H}}^{\prime }=\left( \vec{H}^{\prime }\left(
x\right) \right) _{x\in E}$ is another bundle over $E$, then we call $%
\mathcal{\vec{H}}^{\prime }$ a \underline{subbundle} of $\mathcal{\vec{H}}$
provided $\vec{H}^{\prime }\left( x\right) \subseteq \vec{H}\left( x\right) $
for each $x\in E$. If $\mathcal{\vec{H}}^{\prime }$ is a subbundle of $%
\mathcal{\vec{H}}$, then we write $\mathcal{\vec{H}}\supseteq \mathcal{\vec{H%
}}^{\prime }$. Finally, a \underline{section} of a bundle $\mathcal{\vec{H}}%
=\left( \vec{H}\left( x\right) \right) _{x\in E}$ is an $\mathbb{R}^{d}$%
-valued function $\vec{F}\in C^{m}\left( \mathbb{R}^{n},\mathbb{R}%
^{d}\right) $ such that $J_{x}^{m}\vec{F}\in \vec{H}\left( x\right) $ for
each $x\in E$.

We can now state

\begin{problem}[Generalized Whitney's Extension Problem for $C^{m}$]
\label{problem1} Fix $m,n,d$. How can we decide whether a given bundle $%
\mathcal{\vec{H}}=\left( \vec{H}\left( x\right) \right) _{x\in E}$ has a
section?
\end{problem}

From our discussion of the bundles formed by (\ref{H(x)}) and by (\ref%
{real-value1}), we see that Problem \ref{problem3-1} for $X=C^{m}\left(
\mathbb{R}^{n},\mathbb{R}^{d}\right) $ and Problem \ref{whitneyproblem} for $%
X=C^{m}(\mathbb{R}^n,\mathbb{R})$ are special cases of Problem \ref{problem1}%
.

For the scalar case (i.e., $d=1$), Problem \ref{problem1} is well-understood
thanks to Bierstone-Milman-Paw\l ucki \cite{Bierstone-Milman} and the first
author \cite{F1} (see references therein). Problem \ref{problem1} for $%
X=C^{0}\left( \mathbb{R}^{n},\mathbb{R}^{d}\right) $ is solved in \cite%
{Feff-Kollar}. In this paper, we solve Problem \ref{problem1} for all $m,n,d$
by reducing it to the known scalar case $d=1$.

A variant of Problem \ref{problem1} with $C^{m}(\mathbb{R}^{n},\mathbb{R}%
^{d})$ replaced by $C^{m,\omega }(\mathbb{R}^{n},\mathbb{R}^{d})$ is also of
interest. Here, $C^{m,\omega }(\mathbb{R}^{n},\mathbb{R}^{d})$ is the space
of $C^{m}$ functions whose $m$-th derivatives have a given modulus of
continuity $\omega $ (see Section \ref{notation}). More precisely, we assume
that $\omega $ is a "regular modulus of continuity"\ (again, see Section \ref%
{notation}).

\begin{problem}[Generalized Whitney's Extension Problem for $C^{m,\protect\omega }$]
\label{problem2}Fix $m,n,d$. Let $E$ be an (arbitrary) given subset of $%
\mathbb{R}^{n}$. Suppose at each $x\in E$, we are given an m-jet $\vec{f}%
(x)\in (\mathcal{P}_{m,n})^{d}$ and a convex set $\sigma (x)\subseteq (%
\mathcal{P}_{m,n})^{d}$, with $\sigma \left( x\right) $ symmetric about the
origin. How can decide whether there exist $\vec{F}\in C^{m,\omega }(\mathbb{%
R}^{n},\mathbb{R}^{d})$ and $M<\infty $ such that $J_{x}^{m}\vec{F}-\vec{f}(x)\in
M\sigma (x)$ for all $x\in E$?
\end{problem}

For the case $d=1$, Problem \ref{problem2} has been extensively studied (see
\cite{brudnyi1}, \cite{brudnyi4}, \cite{brudnyi2}, \cite{brudnyi3}, \cite{F2}%
, \cite{F4}, \cite{F5}). More specifically, for $d=1$, if the convex sets $%
\sigma (x)$ are assumed to satisfy a condition called \textquotedblleft
Whitney $\omega $-convexity,\textquotedblright\ then a complete answer to
Problem \ref{problem2} is given in \cite{F2}.

In this paper, we formulate the notion of Whitney $\omega $-convexity for
the general case (all $m,n,d$), and solve Problem \ref{problem2} under the
assumption that the convex sets $\sigma \left( x\right) $ are Whitney $%
\omega $-convex.

We will see that every $\mathcal{R}_{m,n}^{x}$-submodule of $(\mathcal{P}%
_{m,n})^{d}$ is Whitney $\omega $-convex. Consequently, Problem \ref%
{problem2} includes the direct analogue of Problem \ref{problem1} with $%
C^{m}\left( \mathbb{R}^{n},\mathbb{R}^{d}\right) $ replaced by $C^{m,\omega
}\left( \mathbb{R}^{n},\mathbb{R}^{d}\right) $. Thus, by solving Problems %
\ref{problem1} and \ref{problem2} as promised above, we also solve Problem %
\ref{problem3-1} for $X=C^{m}\left( \mathbb{R}^{n},\mathbb{R}^{d}\right) $
and for $X=C^{m,\omega }\left( \mathbb{R}^{n},\mathbb{R}^{d}\right) $.

Our definition of Whitney $\omega $-convexity is given in Section \ref%
{notation}. As we explain there, each Whitney $\omega $-convex set has a
"Whitney constant" $A\geq 1$. In a spirit similar to that of \cite{F2}, we
solve Problem \ref{problem2} for Whitney $\omega $-convex $\sigma (x)$ by
means of the following:

\begin{thm}
\label{Thm1}Fix $m,n,d$. Then there exists an integer $k^{\#}$ (depending
only on integers $m,n,d$) such that the following holds: Let $E\subset
\mathbb{R}^{n}$ be an arbitrary subset and let $\omega $ be a regular
modulus of continuity. Suppose for each $x\in E$, we are given an $m$-jet $%
\vec{f}(x)\in (\mathcal{P}_{m,n})^{d}$ and a Whitney $\omega $-convex set $%
\sigma \left( x\right) \subseteq (\mathcal{P}_{m,n})^{d}$ (with Whitney
constant $A$). Let $M$ be a positive real number. \newline
Assume the following condition is satisfied: For each $S\subseteq E$ with $%
\#\left( S\right) \leq k^{\#}$, there exists $\vec{F}^{S}\in C^{m,\omega
}\left( \mathbb{R}^{n},\mathbb{R}^{d}\right) $ such that

\begin{enumerate}
\item[(i)] $J_{x}^{m}\vec{F}^{S}\in \vec{f}(x)+M\sigma \left( x\right) $ for
all $x\in S$.

\item[(ii)] $\left\Vert \vec{F}^{S}\right\Vert _{C^{m,\omega }\left( \mathbb{%
R}^{n},\mathbb{R}^{d}\right) }\leq M$.
\end{enumerate}

Then there exists $\vec{F}\in C^{m,\omega }\left( \mathbb{R}^{n},\mathbb{R}%
^{d}\right) $ such that $J_{x}^{m}\vec{F}\in \vec{f}(x)+C\left( A\right)
M\sigma \left( x\right) $ for all $x\in E,$ and $\left\Vert \vec{F}%
\right\Vert _{C^{m,\omega }\left( \mathbb{R}^{n},\mathbb{R}^{d}\right) }\leq
C\left( A\right) M$. Here, $C\left( A\right) $ depends only on $A,m,n,d$.
\end{thm}

Theorem \ref{Thm1} is a type of "Finiteness Principle"; the constant $k^{\#}$
is often referred to as a "finiteness constant." The idea of the Finiteness
Principle originated in the work of Brudnyi-Shvartsman (see \cite%
{brudnyi4,pavel0}). In essence, Theorem \ref{Thm1} reduces Problem \ref%
{problem2} for a general set $E$ to the special case of finite $E$ with
bounded cardinality. This special case is readily solvable, as we explain in
Section \ref{solution-to-kollar}.

We are pleasantly surprised to learn that the proof of Theorem \ref{Thm1}
follows from the scalar case ($d=1)$, which has been proven in \cite{F2}. We
should also remark that P. Shvartsman has communicated to us his unpublished
proof of Theorem \ref{Thm1} for the case $m=0$, as a consequence of his
results \cite{pavel1,pavel2} on "Lipschitz selection."

To explain our solution to Problem \ref{problem1}, we need to introduce some
more terminology.

Fix $m,n,d$, and let $k^{\#}$ be a large enough constant depending only on $%
m,n,d$ (see Section \ref{finitenessforcm} for a discussion on the size of $k^{\#}$).

Let $\vec{\mathcal{H}}=\left( \vec{H}\left( x\right) \right) _{x\in E}$ be a
bundle.

Then the "Glaeser refinement" $\vec{\mathcal{H}}^{\prime }=\left( \vec{H}%
^{\prime }\left( x\right) \right) _{x\in E}$ is a subbundle of $\vec{%
\mathcal{H}}$, defined as follows:

Given $x_{0}\in E$ and $\vec{P}_{0}\in \vec{H}\left( x_{0}\right) $ , we say
that $\vec{P}_{0}\in \vec{H}^{\prime }\left( x_{0}\right) $ if and only if
the following condition holds

Given $\varepsilon >0$, there exists $\delta >0$ such that for all $%
x_{1},\ldots ,x_{k^{\#}}\in E\cap B\left( x_{0},\delta \right) $, there
exist $\vec{P}_{1}\in \vec{H}\left( x_{1}\right) ,\ldots ,\vec{P}%
_{k^{\#}}\in \vec{H}\left( x_{k^{\#}}\right) $ with
\begin{equation*}
\left\vert \partial ^{\alpha }\left( \vec{P}_{i}-\vec{P}_{j}\right) \left(
x_{i}\right) \right\vert \leq \varepsilon \left\vert x_{i}-x_{j}\right\vert
^{m-\left\vert \alpha \right\vert },
\end{equation*}%
for $\left\vert \alpha \right\vert \leq m,$ $0\leq i,j\leq k^{\#}$. (Compare
with Glaeser \cite{Glaeser}, Bierstone-Milman-Paw\l ucki \cite%
{Bierstone-Milman}, C. Fefferman \cite{F1}, and C. Fefferman-Koll\'ar \cite%
{Feff-Kollar}.)

The Glaeser refinement has three crucial properties:

\begin{itemize}
\item[(P1)] $\vec{\mathcal{H}}^{\prime }$ is a subbundle of $\vec{\mathcal{H}%
}$.

\item[(P2)] Any section of $\vec{\mathcal{H}}$ is also a section of $\vec{
\mathcal{H}}^{\prime }$. (This follows easily from Taylor's theorem.)

\item[(P3)] In principle, $\vec{\mathcal{H}}^{\prime }$ may be computed from
$\vec{\mathcal{H}}$ by doing elementary linear algebra and computing a $\lim
\sup $. We explain this in Section \ref{solution-to-kollar} below.
\end{itemize}

Let us consider the implication of the Glaeser refinement for Problem \ref%
{problem1}. Starting with a bundle $\vec{\mathcal{H}}_{0}$ over $E$ and
repeatedly taking the Glaeser refinement, we obtain a sequence of bundles
\begin{equation*}
\vec{\mathcal{H}}_{0}\supseteq \vec{\mathcal{H}}_{1}\supseteq \vec{\mathcal{H%
}}_{2}\supseteq \ldots \text{ over }E\text{.}
\end{equation*}%
For each $l$, $\vec{\mathcal{H}}_{l+1}$ is the Glaeser refinement of $\vec{%
\mathcal{H}}_{l}$. By (P1) and (P2), the bundles $\vec{\mathcal{H}}_{l}$
have the same sections. Therefore, Problem \ref{problem1} for a given bundle
$\vec{\mathcal{H}}_{0}$ is the same as Problem \ref{problem1} for any of the
iterated Glaeser refinements $\vec{\mathcal{H}}_{l}$.

A lemma adapted from \cite{F1} (which in turn is adapted from \cite{Glaeser},%
\cite{Bierstone-Milman}) shows that $\vec{\mathcal{H}}_{l}=\left(\vec{H}%
_{l}(x)\right)_{x\in E}$ stablizes after a finite number of iterations. More
precisely, for $L=2\dim [(\mathcal{P}_{m,n})^{d}] +1$, we have $\vec{H}%
_{l}(x)=\vec{H}_{L}(x)$ for all $l\geq L$ and $x \in E$. For the sake of
completeness, we reproduce a proof of this in Section \ref%
{solution-to-kollar} below.

Passing from $\vec{\mathcal{H}}_{0}$ to $\vec{\mathcal{H}}_{L}$, we
therefore reduce Problem \ref{problem1} to the special case in which $\vec{%
\mathcal{H}}$ is its own Glaeser refinement.

This special case of Problem \ref{problem1} is solved by means of the
following:

\begin{thm}
\label{Thm3} Let $\vec{\mathcal{H}}=(\vec{H}(x))_{x\in E}$ be a bundle.
Suppose $\vec{\mathcal{H}}$ is its own Glaeser refinement and each fiber of $%
\vec{\mathcal{H}}$ is non-empty. Then $\vec{\mathcal{H}}$ admits a section.
\end{thm}

In the scalar case $d=1$, Theorem \ref{Thm3} is proven in \cite{F1}. We will
prove Theorem \ref{Thm3} in general by reducing it to the known scalar case,
just as for Theorem \ref{Thm1}.

We shall remark that our methods for solving Problem \ref{problem3-1} for $X=C^{m,\omega}(E)$ and $X=C^{m}(E)$ with $E=\mathbb{R}^n$ apply equally well to the solution of Problem \ref{problem3-1} for $X=C^{m,\omega}(E)$ and $X=C^{m}(E)$ with $E$ being a manifold.

\section*{Acknowledgements}

We are grateful to Arie Israel, Bo'az Klartag, J\'anos Koll\'ar, Assaf Naor,
and Pavel Shvartsman for valuable discussions. We are grateful also to the
American Institute of Mathematics (AIM) and the Office of Naval Research
(ONR) for supporting workshops at which some of those conversations
occurred. We thank J\'anos Koll\'ar for helpful suggestions that led to an improvement of the manuscript. Part of the research was conducted while the second author was
visiting the Mathematical Sciences Center at Tsinghua University, China; the
center's hospitality is gratefully acknowledged.

\section{Notation and Definitions}

\label{notation} We fix integers $d\geq1,n\geq 1$ and $m\geq 0$. $%
C^{m}\left( \mathbb{R}^{n},\mathbb{R}^{d}\right) $ denotes the space of
functions $F:\mathbb{R}^{n}\rightarrow \mathbb{R}^{d}$ whose derivatives up
to order $m$ are continuous and bounded on $\mathbb{R}^{n}$. For $\vec{F}%
=\left( F_{1},\ldots ,F_{d}\right) \in C^{m}\left( \mathbb{R}^{n},\mathbb{R}%
^{d}\right) $, we define the norm
\begin{equation*}
\left\Vert \vec{F}\right\Vert _{C^{m}\left( \mathbb{R}^{n},\mathbb{R}%
^{d}\right) }=\sup_{x\in \mathbb{R}^{n}}\max_{1\leq j\leq d,\left\vert
\alpha \right\vert \leq m}\left\vert \partial ^{\alpha }F_{j}\left( x\right)
\right\vert .
\end{equation*}

$C^{m,\omega }\left( \mathbb{R}^{n},\mathbb{R}^{d}\right) $ denotes the
space of all $C^{m}(\mathbb{R}^{n},\mathbb{R}^{d})$ functions $\vec{F}:%
\mathbb{R}^{n}\rightarrow \mathbb{R}^{d}$ for which the norm
\begin{equation*}
\left\Vert \vec{F}\right\Vert _{C^{m,\omega }\left( \mathbb{R}^{n},\mathbb{R}%
^{d}\right) }=\max \left\{ \left\Vert \vec{F}\right\Vert _{C^{m}\left(
\mathbb{R}^{n},\mathbb{R}^{d}\right) },\sup_{x,y\in \mathbb{R}%
^{n},0<\left\vert x-y\right\vert \leq 1}\max_{1\leq j\leq d,\left\vert
\alpha \right\vert =m}\left\vert \frac{\partial ^{\alpha }F_{j}\left(
x\right) -\partial ^{\alpha }F_{j}\left( y\right) }{\omega \left( \left\vert
x-y\right\vert \right) }\right\vert \right\} 
\end{equation*}
is finite. 

A function $\omega :[0,1]\rightarrow \lbrack 0,\infty )$ is called a
\textquotedblleft regular modulus of continuity\textquotedblright\ if it
satisfies the following conditions:

\begin{itemize}
\item $\omega(0)=\lim_{t\rightarrow0^{+}}\omega(t)=0$ and $\omega(1)=1$.

\item $\omega(t)$ is increasing on $[0,1]$.

\item $\omega(t)/t$ is decreasing on $(0,1]$.
\end{itemize}

Fix $x\in \mathbb{R}^n$. We say that $\sigma(x)\subseteq\left(\mathcal{P}%
_{m,n}\right)^{d}$ is ``Whitney $\omega$-convex (in $\left(\mathcal{P}%
_{m,n}\right)^{d}$) at $x$ with Whitney constant $A$'' if the following
conditions are satisfied:

\begin{itemize}
\item $\sigma(x)$ is closed, convex, symmetric (that is, $\vec{P} \in
\sigma(x)$ if and only if $-\vec{P} \in \sigma(x)$).

\item Let $\left( P_{1},\ldots ,P_{d}\right) \in \sigma(x)$, $Q\in \mathcal{R%
}_{m,n}^{x}$, and $\delta \in (0,1]$. Assume $\left( P_{1},\ldots
,P_{d}\right) $ and $Q$ satisfy the following estimates
\begin{eqnarray*}
\left\vert \partial ^{\alpha }P_{j}\left( x\right) \right\vert &\leq &\omega
\left( \delta \right) \delta ^{m-\left\vert \alpha \right\vert }\text{ for }%
1\leq j\leq d\text{ and }\left\vert \alpha \right\vert \leq m\text{;} \\
\left\vert \partial ^{\alpha }Q\left( x\right) \right\vert &\leq &\delta
^{-\left\vert \alpha \right\vert }\text{ for }\left\vert \alpha \right\vert
\leq m\text{.}
\end{eqnarray*}%
Then
\begin{equation*}
\left( P_{1}\odot _{m,n}^{x}Q,\ldots ,P_{d}\odot _{m,n}^{x}Q\right) \in
A\sigma(x) \text{,}
\end{equation*}%
where $\odot _{m,n}^{x}$ denotes the multiplication in $\mathcal{R}%
_{m,n}^{x} $.
\end{itemize}

From the definition of Whitney $\omega$-convexity, it immediately follows
that every $\mathcal{R}^x_{m,n}$-submodule of $(\mathcal{P}_{m,n})^d$ is
Whitney $\omega$-convex with Whitney constant $1$.

\section{Finiteness Principle for $C^{m,\protect\omega}$}

\label{section3}

In this section, we prove Theorem \ref{Thm1}.

We suppose we are given an arbitrary subset $E\subset \mathbb{R}^{n}$, a
regular modulus of continuity $\omega $, a vector-valued m-jet $\vec{f}%
(x)\in (\mathcal{P}_{m,n}^{x})^{d}$ and a Whitney $\omega $-convex set $%
\sigma \left( x\right) \subseteq (\mathcal{P}_{m,n})^{d}$ (with Whitney
constant $A$) at each point $x\in E$.

We denote by $\hat{x}=\left( \hat{x}_{1},\ldots ,\hat{x}_{n}\right) $ a
dummy variable in $\mathbb{R}^{n}$ and $\hat{v}=\left( \hat{v}_{1},\ldots ,%
\hat{v}_{d}\right) $ a dummy variable in $\mathbb{R}^{d}$.

We note that if $P(\hat{x},\hat{v})$ is an $(m+1)$-jet on $\mathbb{R}^{n+d}$%
, then $[P(\hat{x},\hat{v})]|_{\hat{v}=0}$ is an $(m+1)$-jet on $\mathbb{R}^n$,
and $[\partial_{\hat{v}_j}P]|_{\hat{v}=0}$ is an $m$-jet on $\mathbb{R}^n$.

To prove the theorem, we will show that there exists $G\left(\hat{x},\hat{v}%
\right):\mathbb{R}^{n+d}\rightarrow\mathbb{R}$ such that

\begin{enumerate}
\item $G\left(\hat{x},0\right)\equiv0$.

\item $\left( J_{x}^{m}\left[ \left. \partial _{\hat{v}_{1}}G\left( \hat{x},%
\hat{v}\right) \right\vert _{\hat{v}=0}\right] ,\ldots ,J_{x}^{m}\left[
\left. \partial _{\hat{v}_{d}}G\left( \hat{x},\hat{v}\right) \right\vert _{%
\hat{v}=0}\right] \right) \in \vec{f}(x)+MC\sigma \left( x\right) $ for all $%
x\in E$.

\item $\left\Vert G\right\Vert _{C^{m+1,\omega }\left( \mathbb{R}%
^{n+d}\right) }\leq CM$, for some constant $C$ depending only on $m,n,$ and $%
d$.
\end{enumerate}

Once this is proven, the theorem follows at once by taking
\begin{equation*}
\vec{F}\left( \hat{x}\right) =\left( \left. [\partial _{\hat{v}_{1}}G\left(
\hat{x},\hat{v}\right)] \right\vert _{\hat{v}=0},\ldots ,\left. [\partial _{%
\hat{v}_{d}}G\left( \hat{x},\hat{v}\right)] \right\vert _{\hat{v}=0}\right) .
\end{equation*}

First, we recall the following result (Theorem 3 in \cite{F2}); our proof of
Theorem \ref{Thm1} will be based on it.

\begin{thm}[Finiteness Principle for Real-valued Jets, C. Fefferman \protect\cite{F2}]
\label{Thm2}Given integers $m\geq 0,n\geq 1$, there exists $k^{\#}$,
depending only on $m$ and $n$, such that the following holds:

Let $\omega $ be a regular modulus of continuity, $\bar{E}\subset \mathbb{R}%
^{n}$ an arbitrary subset, and $\bar{A}>0.$ Suppose for each $\bar{x}\in
\bar{E}$, we are given a (real-valued) $m$-jet $\bar{f}(x)\in \mathcal{R}%
_{m,n}^{x}$, and a Whitney $\omega $-convex set $\bar{\sigma}\left( x\right)
\subseteq \mathcal{R}_{m,n}^{x}$ with Whitney constant $\bar{A}$.

Assume the following condition is satisfied: For each $S\subseteq\bar{E}$
with $\#\left(S\right)\leq k^{\#}$, there exists a map $x\longmapsto P^{x}$
from $S\rightarrow\mathcal{P}_{m,n}$ such that (i) $P^{x}\in\bar{f}(x)+M%
\bar{\sigma}\left(x\right)$ for all $x\in S$; (ii) $\left\vert
\partial^{\beta}P^{x}\left(x\right)\right\vert \leq M$ for all $x\in S$, $%
\left\vert \beta\right\vert \leq m$; and (iii) $\left\vert
\partial^{\beta}\left(P^{x}-P^{y}\right)\left(x\right)\right\vert
\leq\omega\left(\left\vert x-y\right\vert \right)\left\vert x-y\right\vert
^{m-\left\vert \beta\right\vert }$ for $\left\vert \beta\right\vert \leq m$
and $\left\vert x-y\right\vert \leq1$, $x,y\in S$.

Then there exists $G\in C^{m,\omega}\left(\mathbb{R}^{n}\right)$ such that
(i) $\left\Vert G\right\Vert _{C^{m,\omega}\left(\mathbb{R}^{n}\right)}\leq
CM$ and (ii) $J_{x}^{m}G\in\bar{f}(x)+CM\bar{\sigma}\left(x\right)$ for all $%
x\in E$. Here, $C$ depends only on $\bar{A}$, $m$, and $n$.
\end{thm}

\begin{remark}\label{finitenesscredit}
In \cite{brudnyi3}, Brudnyi-Shvartsman showed Theorem \ref{Thm2} for $C^{1,\omega}(\mathbb{R}^n)$ and $\bar{\sigma}(x) \equiv 0$ (for all $x \in E$) with the sharp constant $k^{\#}=3 \times 2^{n-1}$. Theorem \ref{Thm2} was first conjectured by Brudnyi and Shvartsman. It was proven by C. Fefferman \cite{F2} with a large constant $k^{\#}$ (depending only on $m$ and $n$). Later, Bierstone-Milman in \cite{Bierstone-Milman0} and Shvartsman in \cite{Pavel} independently improved the constant $k^{\#}$ in the case $\bar{\sigma}(x) \equiv 0$ by showing that $k^{\#}=2^{\dim (\mathcal{P}_{m,n})}$ is sufficient. For Whitney $\omega$-convex sets $\bar{\sigma}$, Shvartsman in \cite{Pavel} also showed that Theorem \ref{Thm2} holds with $k^{\#}=2^{\min \{l+1,\dim \mathcal{P}_{m,n} \}}$, where $l=\max_{x \in \bar{E}}\dim \bar{\sigma}(x)$.
\end{remark}

For each $x\in E$, define
\begin{equation*}
\hat{\sigma}\left( \left( x,0\right) \right) =\left\{ \hat{P}\in \mathcal{R}%
_{m+1,n+d}^{\left( x,0\right) }:\left. \hat{P}\right\vert _{\hat{v}=0}\equiv
0\text{ and }\left( [\partial _{\hat{v}_{1}}\hat{P}]|_{\hat{v}=0},\ldots
,[\partial _{\hat{v}_{d}}\hat{P}]|_{\hat{v}=0}\right) \in \sigma \left(
x\right) \right\} \text{,}
\end{equation*}%
where $\sigma (x)\subseteq (\mathcal{P}_{m,n})^{d}$ is as given in Theorem %
\ref{Thm1}.

\begin{lem}
$\hat{\sigma}\left( \left( x,0\right) \right) $ is Whitney $\omega $-convex
in $\mathcal{R}_{m+1,n+d}^{\left( x,0\right) }$ with Whitney constant $A$.
\end{lem}

\begin{proof}
That $\hat{\sigma}\left( \left( x,0\right) \right) $ is closed, convex, and
symmetric follows directly from the definition of $\hat{\sigma}\left( \left(
x,0\right) \right) $ and the fact that $\sigma \left( x\right) $ is Whitney $%
\omega $-convex.

Suppose $\hat{P}\in \hat{\sigma}\left( \left( x,0\right) \right) $ and $\hat{%
Q}\in \mathcal{R}_{m+1,n+d}^{\left( x,0\right) }$.

Assume
\begin{equation}
\left|\partial^{\gamma}\hat{P}\left(x,0\right)\right|\leq\omega\left(\delta%
\right)\delta^{m+1-\left|\gamma\right|}  \label{31}
\end{equation}
and
\begin{equation}
\left|\partial^{\gamma}\hat{Q}\left(x,0\right)\right|\leq\delta^{-\left|%
\gamma\right|}  \label{32}
\end{equation}
for $\left|\gamma\right|\leq m+1$.

We need to show that $\hat{P}\odot _{m+1,n+d}^{(x,0)}\hat{Q}\in A\hat{\sigma}%
\left( \left( x,0\right) \right) $, where $\odot _{m+1,n+d}^{(x,0)}$ denotes
the multiplication in $\mathcal{R}_{m+1,n+d}^{\left( x,0\right) }$.

First of all, we have
\begin{equation}
\left. \left[\hat{P}\odot _{m+1,n+d}^{(x,0)}\hat{Q}\right]\right\vert _{\hat{v}=0}\equiv
\left. \hat{P}\right\vert _{\hat{v}=0}\odot _{m+1,n}^{x}\left. \hat{Q}%
\right\vert _{\hat{v}=0}\equiv 0,  \label{whitneyconvex1}
\end{equation}%
since $\left. \hat{P}\right\vert _{\hat{v}=0}\equiv 0$ by virtue of the fact
that $\hat{P}\in \hat{\sigma}\left( (x,0)\right) $.

Let $\pi _{m}:\mathcal{R}_{m+1,n+d}^{\left( x,0\right) }\rightarrow \mathcal{%
R}_{m,n+d}^{\left( x,0\right) }$ be the natural projection. We have
\begin{eqnarray}
&&\left. \left( \partial _{\hat{v}_{1}}\left[ \hat{P}\odot _{m+1,n+d}^{(x,0)}%
\hat{Q}\right] ,\ldots ,\partial _{\hat{v}_{d}}\left[ \hat{P}\odot
_{m+1,n+d}^{(x,0)}\hat{Q}\right] \right) \right\vert _{\hat{v}=0}  \notag \\
&=&\left( \partial _{\hat{v}_{1}}\hat{P}|_{\hat{v}=0}\odot _{m,n}^{x}\pi _{m}%
\hat{Q}|_{\hat{v}=0}+\pi _{m}\hat{P}|_{\hat{v}=0}\odot _{m,n}^{x}\partial _{%
\hat{v}_{1}}\hat{Q}|_{\hat{v}=0},\ldots ,\right.  \notag \\
&&\qquad \left. \partial _{\hat{v}_{d}}\hat{P}|_{\hat{v}=0}\odot
_{m,n}^{x}\pi _{m}\hat{Q}|_{\hat{v}=0}+\pi _{m}\hat{P}|_{\hat{v}=0}\odot
_{m,n}^{x}\partial _{\hat{v}_{d}}\hat{Q}|_{\hat{v}=0}\right)  \notag \\
&=&\left( \partial _{\hat{v}_{1}}\hat{P}|_{\hat{v}=0}\odot _{m,n}^{x}\pi _{m}%
\hat{Q}|_{\hat{v}=0},\ldots ,\partial _{\hat{v}_{d}}\hat{P}|_{\hat{v}%
=0}\odot _{m,n}^{x}\pi _{m}\hat{Q}|_{\hat{v}=0}\right) \text{, since }\left.
\hat{P}\right\vert _{\hat{v}=0}\equiv 0\text{. }  \label{0}
\end{eqnarray}

We write $\partial ^{\gamma }=\partial _{\hat{x}}^{\alpha }\partial _{\hat{v}%
}^{\beta }$. If $\beta =1$, from (\ref{31}), we have $\left\vert \partial _{%
\hat{x}}^{\alpha }\partial _{\hat{v}}^{\beta }\hat{P}\left( x,0\right)
\right\vert \leq \omega \left( \delta \right) \delta ^{m-\left\vert \alpha
\right\vert }$. Therefore, we have
\begin{equation}
\left\vert \partial _{\hat{x}}^{\alpha }\left[ \left(\partial _{\hat{v}_{j}}\hat{P}\right)%
|_{\hat{v}=0}\right] \left( x\right) \right\vert \leq \omega \left( \delta
\right) \delta ^{m-\left\vert \alpha \right\vert }\mbox{ for }1\leq j\leq
d,\left\vert \alpha \right\vert \leq m.  \label{33}
\end{equation}

From (\ref{32}), we have
\begin{equation}
\left\vert \partial _{\hat{x}}^{\alpha }\left[ \pi _{m}\hat{Q}|_{\hat{v}=0}%
\right] \left( x\right) \right\vert =\left\vert \partial _{\hat{x}}^{\alpha }%
\hat{Q}\left( x,0\right) \right\vert \leq \delta ^{-\left\vert \alpha
\right\vert }\mbox{ for }\left\vert \alpha \right\vert \leq m.  \label{34}
\end{equation}%
From (\ref{0}), (\ref{33}), (\ref{34}), and the assumption that $\sigma
\left( x\right) $ is Whitney $\omega $-convex (in $\left( \mathcal{P}%
_{m,n}\right) ^{d}$) with Whitney constant $A$, we conclude that
\begin{equation*}
\left( \left. \partial _{\hat{v}_{1}}\left( \hat{P}\odot _{m+1,n+d}^{\left(
x,0\right) }\hat{Q}\right) \right\vert _{\hat{v}=0},\ldots ,\left. \partial
_{\hat{v}_{d}}\left( \hat{P}\odot _{m+1,n+d}^{\left( x,0\right) }\hat{Q}%
\right) \right\vert _{\hat{v}=0}\right) \in A\sigma \left( x\right) .
\end{equation*}%
This together with (\ref{whitneyconvex1}) shows that $\hat{P}\odot
_{m+1,n+d}^{\left( x,0\right) }\hat{Q}\in A\hat{\sigma}\left( \left(
x,0\right) \right) $.

This concludes the proof.
\end{proof}

Given an $\mathbb{R}^{d}$-valued m-jet $\vec{f}\left( x\right) =\left(
f_{1}\left( x\right) ,\ldots ,f_{d}\left( x\right) \right) \in \left(
\mathcal{P}_{m,n}\right) ^{d}$ for each $x\in E$, we define for each point $%
\left( x,0\right) \in E\times \left\{ 0\right\} \subset \mathbb{R}^{n}\times
\mathbb{R}^{d}$ a real-valued $\left( m+1\right) $-jet $\hat{f}\left(
x,0\right) \in \mathcal{R}_{m+1,n+d}^{\left( x,0\right) }$ by
\begin{equation}
\hat{f}{\left( x,0\right) }\left( \hat{x},\hat{v}\right) =\sum_{j=1}^{d}\hat{%
v}_{j}\left[f_{j}\left( x\right) \left( \hat{x}\right) \right]\text{.}
\label{transformedfunction}
\end{equation}

We will check that the assumption in Theorem \ref{Thm2} is satisfied with $%
E\times \{0\}$, $m+1$, $n+d$, $\hat{f}{\left( x,0\right) }$, $\hat{\sigma}%
\left( \left( x,0\right) \right) $ in place of $\bar{E}$, $m,$ $n,$ $\bar{f}%
(x),$ $\bar{\sigma}\left( x\right) $, respectively.

Toward this end, we let $S\times\left\{ 0\right\} \subseteq E\times\left\{
0\right\} $ with $\#\left(S\right)\leq k^{\#}$ and show that there exists a
map $\left(x,0\right)\longmapsto\hat{P}^{\left(x,0\right)}$ from $%
S\times\left\{ 0\right\} \rightarrow\mathcal{R}_{m+1,n+d}^{\left(x,0\right)}$
such that

\begin{itemize}
\item $\hat{P}^{x}-\hat{f}{\left( x,0\right) }\in M\hat{\sigma}\left( \left(
x,0\right) \right) $ for all $x\in S$;

\item $\left\vert \partial^{\gamma}\hat{P}^{\left(x,0\right)}\left(x,0%
\right)\right\vert \leq M$ for all $x\in S$, $\left\vert \gamma\right\vert
\leq m+1$;

\item $\left\vert \partial^{\gamma}\left(\hat{P}^{\left(x,0\right)}-\hat{P}%
^{\left(y,0\right)}\right)\left(x,0\right)\right\vert \leq
M\omega\left(\left\vert x-y\right\vert \right)\left\vert x-y\right\vert
^{m+1-\left\vert \gamma\right\vert }$ for $\left\vert \gamma\right\vert \leq
m+1$ and $\left\vert x-y\right\vert \leq1$, $x,y\in S$.
\end{itemize}

By the assumption of Theorem \ref{Thm1}, there exists $\vec{F}^{S}\in
C^{m,\omega }\left( \mathbb{R}^{n},\mathbb{R}^{d}\right) $ such that (i) $%
J_{x}^{m}\vec{F}^{S}\in f(x)+M\sigma \left( x\right) $ for each $x\in S$ and
(ii) $\left\Vert \vec{F}^{S}\right\Vert _{C^{m,\omega }\left( \mathbb{R}^{n},%
\mathbb{R}^{d}\right) }\leq M$.

For each $x\in S$, we consider the $\mathbb{R}^{d}$-valued $m$-jet $J_{x}^{m}%
\vec{F}^{S}\equiv \left( P_{1}^{x},\ldots ,P_{d}^{x}\right) $ and define a
real-valued $\left( m+1\right) $-jet at $\left( x,0\right) $ by
\begin{equation}
\hat{P}^{\left( x,0\right) }\left( \hat{x},\hat{v}\right) \equiv
\sum_{j=1}^{d}\hat{v}_{j}P_{j}^{x}\left( \hat{x}\right) \text{.}  \label{3.a}
\end{equation}

We will show that $\hat{P}^{\left( x,0\right) }$ defined above satisfies the
three bullet points.

By the definitions of $\hat{f}\left( x,0\right) $ and $\hat{P}^{\left(
x,0\right) }$ (see \eqref{transformedfunction} and \eqref{3.a}), we have
\begin{equation}
\left. \left[ \hat{P}^{\left( x,0\right) }-\hat{f}\left( x,0\right) \right]
\right\vert _{\hat{v}=0}\equiv 0.  \label{6.a}
\end{equation}%
Furthermore, we have
\begin{eqnarray}
&&\left( \left. \left( \partial _{\hat{v}_{1}}\left[ \hat{P}^{\left(
x,0\right) }-\hat{f}{\left( x,0\right) }\right] \right) \right\vert _{\hat{v}%
=0},\ldots ,\left. \left( \partial _{\hat{v}_{d}}\left[ \hat{P}^{\left(
x,0\right) }-\hat{f}{\left( x,0\right) }\right] \right) \right\vert _{\hat{v}%
=0}\right)  \notag \\
&=&\left( P_{1}^{x}-f_{1}^{x},\ldots ,P_{d}^{x}-f_{d}^{x}\right) \in M\sigma
\left( x\right) \text{, since }J_{x}^{m}\vec{F}^{S}\in \vec{f}(x)+M\sigma
\left( x\right) \text{.}  \label{7.a}
\end{eqnarray}%
This together with (\ref{6.a}) proves the first bullet point.

For the second and third bullet points, we write $\partial ^{\gamma
}=\partial _{\hat{x}}^{\alpha }\partial _{\hat{v}}^{\beta }$.

From (\ref{3.a}), we see that
\begin{equation}
\partial _{\hat{v}}^{\beta }\hat{P}^{\left( x,0\right) }|_{\hat{v}=0}\equiv 0%
\mbox{ if}\left\vert \beta \right\vert \not=1.  \label{4.a}
\end{equation}

Thanks to (\ref{4.a}), we have $\left\vert \partial _{\hat{x}}^{\alpha
}\partial _{\hat{v}}^{\beta }\hat{P}^{\left( x,0\right) }\left( x,0\right)
\right\vert =0$ if $\left\vert \beta \right\vert \not=1$; by (\ref{3.a}), we
have $\left\vert \partial _{\hat{x}}^{\alpha }\partial _{\hat{v}}^{\beta }%
\hat{P}^{\left( x,0\right) }\left( x,0\right) \right\vert =\left\vert
\partial _{\hat{x}}^{\alpha }P_{j}^{x}\left( x\right) \right\vert \leq M$
for some $1\leq j \leq d$ if $\left\vert \beta \right\vert =1$ and $%
\left\vert \alpha \right\vert \leq m$. (Here we used the assumption that $%
\left\Vert \vec{F}^{S}\right\Vert _{C^{m,\omega }\left( \mathbb{R}^{n},%
\mathbb{R}^{d}\right) }\leq M$.) To wit, we have $\left\vert \partial
^{\gamma }\hat{P}^{\left( x,0\right) }\left( x,0\right) \right\vert \leq M$
for all $x\in S$, $\left\vert \gamma \right\vert \leq m+1$. This shows the
second bullet point.

To prove the third bullet point, we consider two cases: $\left\vert \beta
\right\vert \not=1$ and $\left\vert \beta \right\vert =1$. For $x,y\in S$
and $\left\vert \beta \right\vert \not=1$, in view of (\ref{4.a}), we see
that
\begin{equation}
\left\vert \partial ^{\gamma }\left( \hat{P}^{\left( x,0\right) }-\hat{P}%
^{\left( y,0\right) }\right) \left( x,0\right) \right\vert =0\text{,}
\label{w-1}
\end{equation}%
which trivially implies the third bullet point. Now, for $\left\vert \beta
\right\vert =1$ and for all $x,y\in S$ with $|x-y|\leq 1$, we have
\begin{eqnarray}
&&\left\vert \partial ^{\gamma }\left( \hat{P}^{\left( x,0\right) }-\hat{P}%
^{\left( y,0\right) }\right) \left( x,0\right) \right\vert  \notag \\
&=&\left\vert \partial _{\hat{x}}^{\alpha }\left( P_{j}^{x}-P_{j}^{y}\right)
\left( x\right) \right\vert \text{ for some }1\leq j\leq d,  \notag \\
&\leq &M\omega \left( \left\vert x-y\right\vert \right) \left\vert
x-y\right\vert ^{m-\left\vert \alpha \right\vert }\text{, since }\left\Vert
\vec{F}^{S}\right\Vert _{C^{m,\omega }\left( \mathbb{R}^{n},\mathbb{R}%
^{d}\right) }\leq M,  \notag \\
&\leq &M\omega \left( \left\vert x-y\right\vert \right) \left\vert
x-y\right\vert ^{m+1-\left\vert \gamma \right\vert }\text{, since }%
\left\vert \gamma \right\vert =1+\left\vert \alpha \right\vert \text{ when }%
\left\vert \beta \right\vert =1\text{.}  \label{w-2}
\end{eqnarray}
From (\ref{w-1}) and (\ref{w-2}), we obtain the third bullet point.

We have verified the three bullet points; by Theorem \ref{Thm2}, we can
conclude that there exists $G\in C^{m+1,\omega}\left(\mathbb{R}^{n+d}\right)$
such that

\begin{enumerate}
\item $\left\Vert G\right\Vert _{C^{m+1,\omega}\left(\mathbb{R}%
^{n+d}\right)}\leq CM$.

\item $J_{\left( x,0\right) }^{m+1}G-\hat{f}{\left( x,0\right) }\in CM\hat{%
\sigma}\left( \left( x,0\right) \right) $ for all $x\in E$.
\end{enumerate}

Here, $C$ depends only on $A$, $m$, $n$, and $d$. By the definition of $\hat{%
\sigma}\left( \left( x,0\right) \right) $ and (\ref{transformedfunction}),
we see that $G\left( \hat{x},\hat{v}\right) :\mathbb{R}^{n+d}\rightarrow
\mathbb{R}$ satisfies

\begin{enumerate}
\item $G\left(\hat{x},0\right)\equiv0$.

\item $\left(J_{x}^{m}[\partial_{\hat{v}_{1}}G|_{\hat{v}_1=0}],%
\ldots,J_{x}^{m}[\partial_{\hat{v}_{d}}G|_{\hat{v}=0}]\right)\in
\vec{f}(x)+MC\sigma\left(x\right)$ for all $x\in E$.

\item $\left\Vert G\right\Vert _{C^{m+1,\omega}\left(\mathbb{R}%
^{n+d}\right)}\leq CM$, for some constant $C$ depending only on $m$, $n$, $d$%
, and $A$.
\end{enumerate}

In view of the remarks at the beginning of this section, we have proven
Theorem \ref{Thm1}.

\section{Proof of the $C^{m}$ Extension Theorem}

\label{finitenessforcm} In this section, we prove Theorem \ref{Thm3}. The
relevant terminology is given in the introduction.

Before we embark on the proof of Theorem \ref{Thm3}, we recall the following
result (Theorem 2 in \cite{F1}).

\begin{thm}[$C^{m}$ Extension Theorem for Real-valued Functions,
C. Fefferman\protect\cite{F1}]
\label{feff-thm1} Let $\bar{E}\subset \mathbb{R}^{n}$ be compact. Suppose
for each $x\in\bar{E}$ we are given an affine subspace $\bar{H}%
\left(x\right)\subseteq\mathcal{R}_{m,n}^{x}$ having the form $\bar{H}%
\left(x\right)=\bar{f}\left(x\right)+\bar{I}\left(x\right)$, where $\bar{f}%
\left(x\right)\in\mathcal{R}_{m,n}^{x}$ and $\bar{I}\left(x\right)$ is an
ideal in $\mathcal{R}_{m,n}^{x}$. Assume that $\{\bar{H}\left(x\right)\}_{x\in \bar{E}}$ is its
own Glaeser refinement. Then there exists $\bar{F}%
\in C^{m}\left(\mathbb{R}^{n}\right)$ with $J_{x}^{m}\bar{F}\in\bar{H}%
\left(x\right)$ for all $x\in E$.
\end{thm}

\begin{remark}\label{sharpconstant}
In \cite{F1}, C. Fefferman showed that the large constant $k^{\#}$ appearing implicitly in Theorem \ref{feff-thm1} can be bounded by a constant depending only on $m$ and $n$. Later, in \cite{Bierstone-Milman0}, Bierstone-Milman gave a sharper upper bound on $k^{\#}$ in the case $\bar{I}(x) \equiv \{ 0 \}$: In that case, they showed that $k^{\#} = 2^{\dim (\mathcal{P}_{m,n})}$ is sufficient. Compare with Remark \ref{finitenesscredit}.
\end{remark}

\begin{proof}[Proof of Theorem \protect\ref{Thm3}]
We recall that each fiber of the given bundle $\vec{\mathcal{H}}=(\vec{H}%
(x))_{x\in E}$ takes the form
\begin{equation*}
\vec{H}(x)=\vec{P}^{x}+\vec{I}(x),
\end{equation*}%
where $\vec{P}^{x}=(P_{1}^{x},\ldots ,P_{d}^{x})\in \left( \mathcal{P}%
_{m,n}\right) ^{d}$ and $\vec{I}(x)$ is an $\mathcal{R}_{m,n}^{x}$-submodule
of $(\mathcal{P}_{m,n})^{d}$. Define $\hat{f}:E\times \left\{ 0\right\}
\subset \mathbb{R}^{n}\times \mathbb{R}^{d}\rightarrow \mathcal{P}_{m+1,n+d}$
by
\begin{equation}
\hat{f}\left( \left( x,0\right) \right) =\sum_{j=1}^{d}\hat{v}%
_{j}P_{j}^{x}\left( \hat{x}\right) \in \mathcal{R}_{m+1,n+d}^{\left(
x,0\right) }.  \label{ff}
\end{equation}

For each $\left( x,0\right) \in E\times \left\{ 0\right\} $, consider the
set
\begin{equation*}
\hat{I}\left( \left( x,0\right) \right) =\left\{ P\in \mathcal{R}%
_{m+1,n+d}^{\left( x,0\right) }:P\left( \hat{x},0\right) \equiv 0,\left( %
\left[ \partial _{\hat{v}_{1}}P\left( \hat{x},\hat{v}\right) \right] |_{\hat{%
v}=0},\ldots ,\left[ \partial _{\hat{v}_{d}}P\left( \hat{x},\hat{v}\right) %
\right] |_{\hat{v}=0}\right) \in \vec{I}\left( x\right) \right\} .
\end{equation*}

\begin{lem}
\label{lem2} For each $\left(x,0\right)\in E\times\left\{ 0\right\} $, $\hat{%
I}\left(\left(x,0\right)\right)$ is an ideal in $\mathcal{R}%
_{m+1,n+d}^{\left(x,0\right)}$.
\end{lem}

\begin{proof}[Proof of Lemma \protect\ref{lem2}]
Let $P\in \hat{I}\left( \left( x,0\right) \right) $ and $Q\in \mathcal{R}%
_{m+1,n+d}^{\left( x,0\right) }$. We must show that
\begin{equation}
\left[ P\odot _{m+1,n+d}^{\left( x,0\right) }Q\right] \in \hat{I}\left(
\left( x,0\right) \right) ,  \label{10}
\end{equation}%
where $\odot _{m+1,n+d}^{\left( x,0\right) }$ denotes the multiplication in $%
\mathcal{R}_{m+1,n+d}^{\left( x,0\right) }$.

Since $P\left( \hat{x},0\right) \equiv 0$, we have
\begin{equation}
\left. \left[ P\odot _{m+1,n+d}^{\left( x,0\right) }Q\right] \right\vert _{%
\hat{v}=0}\equiv P|_{\hat{v}=0}\odot _{m+1,n}^{x }Q|_{\hat{v}=0}\equiv 0.
\label{8}
\end{equation}

Now, let $\pi _{m}:\mathcal{R}_{m+1,n+d}^{\left( x,0\right) }\rightarrow
\mathcal{R}_{m,n+d}^{\left( x,0\right) }$ be the natural projection. We have
\begin{align}
& \left( \left. \partial _{\hat{v}_{1}}\left( P\odot _{m+1,n+d}^{\left(
x,0\right) }Q\right) \right\vert _{\hat{v}=0},\ldots ,\left. \partial _{\hat{%
v}_{d}}\left( P\odot _{m+1,n+d}^{\left( x,0\right) }Q\right) \right\vert _{%
\hat{v}=0}\right)  \notag \\
=& (\left. \partial _{\hat{v}_{1}}P\right\vert _{\hat{v}=0}\odot
_{m,n}^{x}\left. \pi _{m}Q\right\vert _{\hat{v}=0}+\left. \pi
_{m}P\right\vert _{\hat{v}=0}\odot _{m,n}^{x}\left. \partial _{\hat{v}%
_{1}}Q\right\vert _{\hat{v}=0},\ldots ,  \notag \\
& \qquad \left. \partial _{\hat{v}_{d}}P\right\vert _{\hat{v}=0}\odot
_{m,n}^{x}\left. \pi _{m}Q\right\vert _{\hat{v}=0}+\left. \pi
_{m}P\right\vert _{\hat{v}=0}\odot _{m,n}^{x}\left. \partial _{\hat{v}%
_{d}}Q\right\vert _{\hat{v}=0})  \notag \\
=& \left( \partial _{\hat{v}_{1}}P|_{\hat{v}=0}\odot _{m,n}^{x}\pi _{m}Q|_{%
\hat{v}=0},\ldots ,\partial _{\hat{v}_{d}}P|_{\hat{v}=0}\odot _{m,n}^{x}\pi
_{m}Q|_{\hat{v}=0}\right) \in \vec{I}\left( x\right) ,  \label{9}
\end{align}%
where we used the facts that $P\in \hat{I}\left( \left( x,0\right) \right) $
(so that $\pi_m P|_{\hat{v}=0}\equiv 0$) and that $\vec{I}\left( x\right) $
is an $\mathcal{R}_{m,n}^{x}$-submodule of $\left( \mathcal{P}%
_{m,n}^{x}\right) ^{d}$.

From (\ref{8}) and (\ref{9}), we conclude the proof of (\ref{10}).
\end{proof}

\begin{remark}
\label{rmk2} 
We recall that $\left\{ \vec{H}\left( x\right) \right\} _{x\in E}$
is assumed to be its own Glaeser refinement. We will show that the following bundle 

\begin{equation}
\left\{ \widehat{H}\left( \left( x,0\right) \right) :=\hat{f}\left( \left(
x,0\right) \right) +\hat{I}\left( \left( x,0\right) \right) \right\}
_{\left( x,0\right) \in E\times \left\{ 0\right\} }\text{ is also its own Glaeser
refinement}.  \label{4}
\end{equation}%

With $n+d$, $m+1$, $E\times \left\{ 0\right\} $, $\hat{f}$,
$\hat{I}\left( \left( x,0\right) \right) $,  and $\left\{ \widehat{H}\left( \left( x,0\right) \right) \right\}_{\left( x,0\right) \in E\times \left\{ 0\right\}}$ in place of $n$, $m$, $\bar{E}$, $%
\bar{f}$, $\bar{I}\left( x\right) $, and $\left\{ \bar{H}(x) \right\}_{x\in \bar{E}}$in Theorem \ref{feff-thm1}, we see that there exists $G\in C^{m+1}\left(
\mathbb{R}^{n+d}\right) $ with $J_{\left( x,0\right) }^{m+1}\left( G\right)
\in \widehat{H}\left( \left( x,0\right) \right) $ for all $\left( x,0\right)
\in E\times \left\{ 0\right\} $.

Setting
\begin{equation*}
\vec{F}\left( \hat{x}\right) =\left( \left. [\partial _{\hat{v}_{1}}G\left(
\hat{x},\hat{v}\right) ]\right\vert _{\hat{v}=0},\ldots ,\left. [\partial _{%
\hat{v}_{d}}G\left( \hat{x},\hat{v}\right) ]\right\vert _{\hat{v}=0}\right) ,
\end{equation*}%
in view of the definition of $\widehat{H}\left( \left( x,0\right) \right) $
, we have $J_{x}^{m}\vec{F}\in \vec{H}(x)$ for all $x\in E$, thus proving
Theorem \ref{Thm3}.
\end{remark}

Therefore, the crux of the matter is to verify (\ref{4}). To this end, let $%
x_{0}\in E$ and $\hat{P}_{0}\in \widehat{H}\left( \left( x_{0},0\right)
\right) $. We will prove that $\hat{P}_{0}\in \widehat{H}^{\prime }\left(
\left( x_{0},0\right) \right) $.

Since $\hat{P}_{0}\in \widehat{H}\left( \left( x_{0},0\right) \right) $, we
can write
\begin{equation}
\hat{P}_{0}\left( \hat{x},\hat{v}\right) =\hat{f}\left( \left(
x_{0},0\right) \right) (\hat{x},\hat{v})+\sum_{1\leq \left\vert \xi
\right\vert \leq m+1}\frac{1}{\xi !}\hat{v}^{\xi }P_{\xi }\left( \hat{x}%
\right) \text{.}  \label{ideal1}
\end{equation}%
where $P_{\xi }\in \mathcal{P}_{m,n}$ with
\begin{equation}
\vec{P}\equiv \left(\left. \left( \partial _{\hat{v}_{1}}\left[ \sum_{1\leq
\left\vert \xi \right\vert \leq m+1}\frac{1}{\xi !}\hat{v}^{\xi }P_{\xi
}\left( \hat{x}\right) \right] \right) \right\vert _{\hat{v}=0},\ldots
,\left. \left( \partial _{\hat{v}_{d}}\left[ \sum_{1\leq \left\vert \xi
\right\vert \leq m+1}\frac{1}{\xi !}\hat{v}^{\xi }P_{\xi }\left( \hat{x}%
\right) \right] \right) \right\vert _{\hat{v}=0}\right)\in \vec{I}\left( x\right)
\text{.}  \label{ideal4}
\end{equation}

We write
\[
\vec{P}_{0}\equiv \left( P_{0,1},\ldots ,P_{0,d}\right) \equiv \vec{P}%
^{x_{0}}+\vec{P}.
\]

Thanks to (\ref{ideal4}), we have
\begin{equation}
\vec{P}_{0} \in \vec{H}(x_0). \label{ideal6}
\end{equation}%

From (\ref{ideal1}), we have
\begin{equation}
\hat{P}_{0}\left( \hat{x},\hat{v}\right) =\sum_{1\leq k\leq d}\hat{v}%
_{k}P_{0,k}\left( \hat{x}\right) +\sum_{2\leq \left\vert \xi \right\vert
\leq m+1}\frac{1}{\xi !}\hat{v}^{\xi }P_{\xi }\left( \hat{x}\right) \text{.}
\label{ideal7}
\end{equation}


Fix $\varepsilon>0$.

Since $\left\{ \vec{H}\left( x\right) =\vec{P}^{x}+\vec{I}\left( x\right)
\right\} _{x\in E}$ is its own Glaeser refinement, we know that there exists
$\delta >0$ such that for all $x_{1},\ldots ,x_{k^{\#}}\in E\cap B\left(
x_{0},\delta \right) $, there exist
\begin{equation}
\vec{P}_{1}=\left( P_{1,1},\ldots ,P_{1,d}\right) \in \vec{H}\left(
x_{1}\right) ,\ldots ,\vec{P}_{k^{\#}}=\left( P_{k^{\#},1},\ldots
,P_{k^{\#},d}\right) \in \vec{H}\left( x_{k^{\#}}\right)  \label{1}
\end{equation}%
with
\begin{equation}
\left\vert \partial ^{\alpha }\left( P_{i,k}-P_{j,k}\right) \left(
x_{i}\right) \right\vert \leq \varepsilon \left\vert x_{i}-x_{j}\right\vert
^{m-\left\vert \alpha \right\vert }  \label{2}
\end{equation}%
for $\left\vert \alpha \right\vert \leq m,0\leq i,j\leq k^{\#},1\leq k\leq d$. 

Now, for any $\left(x_{1},0\right),\ldots,\left(x_{k^{\#}},0\right)\in%
\left(E\times\left\{ 0\right\} \right)\cap
B\left(\left(x_0,0\right),\delta\right)$, we claim the following:

\begin{itemize}
\item
\begin{eqnarray*}
\hat{P}_{1}\left( \hat{x},\hat{v}\right) &\equiv &\sum_{j=1}^{d}\hat{v}%
_{j}P_{1,j}\left( \hat{x}\right) +\sum_{2\leq \left\vert \xi \right\vert
\leq m+1}\frac{1}{\xi !}\hat{v}^{\xi }P_{\xi }\left( \hat{x}\right) \in
\widehat{H}\left( \left( x_{1},0\right) \right) , \\
&&\ldots, \\
\hat{P}_{k^{\#}}\left( \hat{x},\hat{v}\right) &\equiv &\sum_{j=1}^{d}\hat{v}%
_{j}P_{k^{\#},j}\left( \hat{x}\right) +\sum_{2\leq \left\vert \xi
\right\vert \leq m+1}\frac{1}{\xi !}\hat{v}^{\xi }P_{\xi }\left( \hat{x}%
\right) \in \widehat{H}\left( \left( x_{k^{\#}},0\right) \right) ,
\end{eqnarray*}%
where $\vec{P}_{1},\ldots ,\vec{P}_{k^{\#}}$ are as chosen in (\ref{1}); $%
P_{\xi }$ are as in (\ref{ideal1}).

\item For all $\left\vert \gamma \right\vert \leq m+1$, we have
\begin{equation}
\left\vert \partial ^{\gamma }\left( \hat{P}_{i}-\hat{P}_{j}\right) \left(
\left( x_{i},0\right) \right) \right\vert \leq \varepsilon \left\vert
x_{i}-x_{j}\right\vert ^{m+1-\left\vert \gamma \right\vert },\mbox{ for }%
0\leq i,j\leq k^{\#}.  \label{7}
\end{equation}
\end{itemize}

To see the first bullet point, fix an integer $i\in \left\{ 1,\ldots
,k^{\#}\right\} $ and consider
\begin{align}
& \hat{P}_{i}\left( \hat{x},\hat{v}\right) -\hat{f}\left( \left(
x_{i},0\right) \right) \left( \hat{x},\hat{v}\right)  \notag \\
=& \left( \sum_{j=1}^{d}\hat{v}_{j}P_{i,j}\left( \hat{x}\right) +\sum_{2\leq
\left\vert \xi \right\vert \leq m+1}\frac{1}{\xi !}\hat{v}^{\xi }P_{\xi
}\left( \hat{x}\right) \right) -\sum_{j=1}^{d}\hat{v}_{j}P_{j}^{x_{i}}\left(
\hat{x}\right)  \notag \\
=& \sum_{j=1}^{d}\hat{v}_{j}\left( P_{i,j}-P_{j}^{x_{i}}\right) \left( \hat{x%
}\right) +\sum_{2\leq \left\vert \xi \right\vert \leq m+1}\frac{1}{\xi !}%
\hat{v}^{\xi }P_{\xi }\left( \hat{x}\right) .  \label{5}
\end{align}%
Notice that
\begin{equation}
\left. \left[ \partial _{\hat{v}}^{\beta }\left( \sum_{2\leq \left\vert \xi
\right\vert \leq m+1}\frac{1}{\xi !}\hat{v}^{\xi }P_{\xi }\left( \hat{x}%
\right) \right) \right] \right\vert _{\hat{v}=0}\equiv 0\text{ for }%
\left\vert \beta \right\vert \leq 1\text{.}  \label{ideal2}
\end{equation}%
From (\ref{5}) and (\ref{ideal2}), it follows that
\begin{equation}
\left[ \hat{P}_{i}\left( \hat{x},0\right) -\hat{f}\left( \left(
x_{i},0\right) \right) \right] \left( \hat{x},0\right) \equiv 0.  \label{6}
\end{equation}

Next, we have
\begin{align*}
& \left( \left. \left[ \partial _{\hat{v}_{1}}\left( \hat{P}_{i}-\hat{f}%
\left( \left( x_{i},0\right) \right) \right) \left( \hat{x},\hat{v}\right) %
\right] \right\vert _{\hat{v}=0},\ldots ,\left. \left[ \partial _{\hat{v}%
_{d}}\left( \hat{P}_{i}-\hat{f}\left( \left( x_{i},0\right) \right) \right)
\left( \hat{x},\hat{v}\right) \right] \right\vert _{\hat{v}=0}\right) \\
=& \left( \left( P_{i,1}-P_{1}^{x_{i}}\right) \left( \hat{x}\right) ,\ldots
,\left( P_{i,d}-P_{d}^{x_{i}}\right) \left( \hat{x}\right) \right) \in \vec{I%
}\left( x_{i}\right) ,
\end{align*}%
where the first equality follows from (\ref{5}) and (\ref{ideal2}); and the
last relation follows from (\ref{1}). Together with (\ref{6}), this
completes the proof of the first bullet point.

For the second bullet point, we write $\partial ^{\gamma }=\partial _{\hat{x}%
}^{\alpha }\partial _{\hat{v}}^{\beta }$ and observe that for $0\leq i\leq
k^{\#}$,%
\begin{equation}
\left. \left[ \partial _{\hat{v}}^{\beta }\hat{P}_{i}\right] \right\vert _{%
\hat{v}=0}\equiv \left\{
\begin{array}{c}
0 \\
P_{i,j} \\
P_{\beta }%
\end{array}%
\right.
\begin{array}{l}
\text{if }\left\vert \beta \right\vert =0 \\
\text{for some }1\leq j\leq d\text{, if }\left\vert \beta \right\vert =1 \\
\left\vert \beta \right\vert \geq 2%
\end{array}.  \label{observation1}
\end{equation}%
In view of (\ref{observation1}), we see that (\ref{7}) holds trivially for $%
\left\vert \beta \right\vert \not=1$. Therefore, it suffices to show (\ref{7}%
) for $\left\vert \beta \right\vert =1$. Without loss of generality, we may
assume $\partial _{\hat{v}}^{\beta }=\partial _{\hat{v}_{k}}$ for some $k\in
\left\{ 1,\ldots ,d\right\} $. We have
\begin{align*}
& \left\vert \partial ^{\gamma }\left( \hat{P}_{i}-\hat{P}_{j}\right) \left(
\left( x_{i},0\right) \right) \right\vert  \\
=& \left\vert \partial _{\hat{x}}^{\alpha }\partial _{\hat{v}_{k}}\left(
\hat{P}_{i}-\hat{P}_{j}\right) \left( \left( x_{i},0\right) \right)
\right\vert  \\
=& \left\vert \partial _{\hat{x}}^{\alpha }\left( P_{i,k}-P_{j,k}\right)
\left( x_{j}\right) \right\vert  \\
\leq & \varepsilon \left\vert x_{i}-x_{j}\right\vert ^{m-\left\vert \alpha
\right\vert },\mbox{
thanks to \ensuremath{\left(\ref{2}\right)}}, \\
=& \varepsilon \left\vert x_{i}-x_{j}\right\vert ^{m+1-\left\vert \gamma
\right\vert },
\end{align*}%
establishing (\ref{7}).

The two bullet points show that $\widehat{P}_0 \in \widehat{H}'\left( \left( x_{0},0\right)
\right) $. Since $x_{0}\in E$ and $\widehat{P}_0 \in \widehat{H}\left( \left( x_{0},0\right)
\right)$ are arbitrary,
this completes the proof of (\ref{4}). Theorem \ref{Thm3} now follows from
Remark \ref{rmk2}.
\end{proof}

\section{Solutions to Problem \protect\ref{problem3-1} for $X=C^m(\mathbb{R}%
^n)$ and $X=C^{m,\protect\omega}(\mathbb{R}^n)$}

\label{solution-to-kollar}

Armed with Theorems \ref{Thm1} and \ref{Thm3}, we are now in a position to
answer Problem \ref{problem3-1} for $X=C^m(\mathbb{R}^n)$ and $%
X=C^{m,\omega}(\mathbb{R}^n)$.

We write $Q^{o}=[-1/2,1/2]^{n}$ to denote the unit cube in $\mathbb{R}^{n}$.

For $X=C^{m,\omega}(Q^{o})$, to apply Theorem \ref{Thm1}, we will take $E=Q^{o}$
and
\begin{equation*}
\sigma(x)=\left\{ P\in(\mathcal{P}_{m,n})^{d}:%
\sum_{i=1}^{d}P_{i}(x)f_{i}(x)=0\right\} .
\end{equation*}

It is easy to see that $\sigma(x)$ is Whitney $\omega$-convex with Whitney
constant $1$. By Theorem \ref{Thm1} and the standard Whitney extension Theorem for $C^{m, \omega}(\mathbb{R}^n)$ (see Theorem \ref{Thm2}), we easily see that the solvability of
Problem \ref{problem3-1} for $X=C^{m,\omega}(Q^{o})$ is equivalent to the
solvability of the following elementary linear algebra problem:

\begin{enumerate}
\item \emph{Does there exist $M<\infty$ such that the following holds:
Given any $k^{\#}$ distinct points $x_1,\ldots,x_{k^{\#}}\in Q^o$, there
exist $P_1^j,\ldots,P^j_{k^{\#}}\in \mathcal{P}_{m,n}$ for $1\leq j\leq d$
such that
\begin{equation*}
\left\{
\begin{array}{l}
\sum_{j=1}^{d}f_{j}\left( x_{k}\right) P_{k}^{j}\left( x_{k}\right) =\phi
\left( x_{k}\right) \text{ for }k=1,\ldots ,k^{\#} \\
\sum_{k=1}^{k^{\#}}\sum_{j=1}^{d}\sum_{\left\vert \alpha \right\vert \leq
m}\left\vert \partial ^{\alpha }P_{k}^{j}\left( x_{k}\right) \right\vert
^{2}+\sum_{1\leq k<k^{\prime }\leq k^{\#}}\sum_{j=1}^{d}\sum_{\left\vert
\alpha \right\vert \leq m}\left\vert \frac{\partial ^{\alpha }\left(
P_{k}^{j}-P_{k^{\prime }}^{j}\right) \left( x_{k}\right) }{\omega \left(
\left\vert x_{k}-x_{k^{\prime }}\right\vert \right) \left\vert
x_{k}-x_{k^{\prime }}\right\vert ^{m-\left\vert \alpha \right\vert }}%
\right\vert^2 \leq M^{2}%
\end{array}%
\right. \text{.}
\end{equation*}
}
\end{enumerate}

Next, we describe the solution to Problem \ref{problem3-1} for $X=C^m(Q^o)$.

It is easy to see that the m-jet of any $C^{m}\left( Q^{o}\right) $ solution of Problem 1 at $x\in
Q^{o}$ belongs to
\begin{equation*}
\vec{H}\left( x\right) =\left\{ \vec{P}=\left( P_{1},\ldots ,P_{d}\right)
\in \left( \mathcal{P}_{m,n}\right) ^{d}:\sum_{i=1}^{d}P_{i}\left( x\right)
f_{i}\left( x\right) =\phi \left( x\right) \right\} \text{.}
\end{equation*}%
We consider the bundle $\mathcal{\vec{H}}=\left( \vec{H}\left( x\right)
\right)_{x\in Q^o}$. Solving Problem \ref{problem3-1} amounts to deciding whether this bundle admits a
section.

By Taylor's theorem, we know that the sections of a bundle $\mathcal{\vec{H}}$ coincide with the sections of its Glaeser refinement.

Now we describe an effective method for computing the Glaeser refinements.

Given points $x_{0},x_{1},\ldots ,x_{k}\in \mathbb{R}^{n}$, and given
polynomials $\vec{P}_{0}=\left( \vec{P}_{0,1},\ldots ,\vec{P}_{0,d}\right) ,%
\vec{P}_{1}=\left( \vec{P}_{1,1},\ldots ,\vec{P}_{1,d}\right) ,\ldots ,\vec{P%
}_{k}=\left( \vec{P}_{k,1},\ldots ,\vec{P}_{k,d}\right) \in \left( \mathcal{P%
}_{m,n}\right) ^{d}$, we define
\begin{eqnarray*}
&&\mathcal{Q}\left( \vec{P}_{0},\vec{P}_{1},\ldots ,\vec{P}%
_{k};x_{0},x_{1},\ldots ,x_{k}\right) \\
&\equiv &\sum_{\substack{ i^{\prime },i=0  \\ x_{i}\not=x_{i^{\prime }}}}%
^{k}\sum_{j=1}^{d}\sum_{\left\vert \alpha \right\vert \leq m}\left\vert
\frac{\partial ^{\alpha }\left( P_{i,j}-P_{i^{\prime },j}\right) \left(
x_{i}\right) }{\left\vert x_{i}-x_{i^{\prime }}\right\vert ^{m-\left\vert
\alpha \right\vert }}\right\vert ^{2}\text{.}
\end{eqnarray*}

Fix $x_0 \in Q^o$ and $\vec{P}_0 \in \vec{H}(x_0)$. For fixed $x_1,\ldots,
x_k^{\#}\in Q^o$, we compute the minimum for the following quadratic form
over a finite-dimensional affine space:
\begin{eqnarray*}
&&MIN(x_0,\vec{P}_0;x_1,\ldots,x_k^{\#}) \\
&\equiv& \min \{ \mathcal{Q}(\vec{P}_0,\vec{P}_1,\ldots,\vec{P}%
_{k^{\#}};x_0,x_1,\ldots,x_{k^{\#}}): \vec{P}_1\in \vec{H}(x_1),\ldots,\vec{P%
}_{k^{\#}} \in \vec{H}(x_{k^{\#}})\}.
\end{eqnarray*}
Determining $MIN(x_0,\vec{P}_0;x_1,\ldots,x_k^{\#})$ is just a routine linear algebra problem. From the definition of Glaeser refinement, it is
easy to see that given $x_0 \in Q^o$ and $\vec{P}_0 \in \vec{H}(x_0)$, we have $\vec{P}_0 \in \vec{H}%
^{\prime }(x_0)$ if and only if $\limsup_{\delta \downarrow 0,
x_1,\ldots,x_k^{\#} \in B(x_0,\delta)} MIN(x_0,\vec{P}_0;x_1,\ldots,x_k^{%
\#})=0$. To wit, computing the Glaeser refinement of a bundle $\vec{\mathcal{%
H}}$ involves doing elementary linear algebra and calculating the $\limsup$.

The following lemma states that the Glaeser refinements stablize after a
finite number of iterations. 


\begin{lem}[Stabilization Lemma adapted from \protect\cite{F1} (which in
turn is adapted from \cite{Glaeser,Bierstone-Milman})]

{\label{lem1}} Let $E\subset \mathbb{R}^n$ be compact. Suppose we are given a bundle $\vec{\mathcal{H}}_0=\{\vec{H}(x)\}_{x\in E}$. For $l \geq 0$, let $\vec{\mathcal{H}}_{l+1}=\{\vec{H}_{l+1}(x) \}_{x \in E}$ be the Glaeser refinement of $\vec{\mathcal{H}}_l=\{\vec{H}_l(x)\}_{x \in E}$. For each $x \in E$, if $\dim \vec{H}_{2k+1}\left(x\right)\geq\dim%
\left[\left(\mathcal{P}_{m,n}\right)^{d}\right]-k$, then $\vec{H}%
_{l}\left(x\right)=\vec{H}_{2k+1}\left(x\right)$ for all $l\geq2k+1$.
\end{lem}

\begin{proof}[Proof of Lemma \protect\ref{lem1}]
Fix $x \in E$. We proceed by induction on $k$. For $k=0$, the lemma asserts that 
\begin{equation}
\mbox{if
\ensuremath{\vec{H}_{1}\left(x\right)=\left(\mathcal{P}_{m,n}\right)^{d}}, then
\ensuremath{\vec{H}_{l}\left(x\right)=\vec{H}_{1}\left(x\right)} for all
\ensuremath{l\geq1}}.  \label{11}
\end{equation}
From the definition of $\vec{\mathcal{H}}_l$, one sees that
\begin{equation}
\dim \vec{H}_{l+1}\left(x\right)\leq\lim_{y\rightarrow x}\inf\dim \vec{H}%
_{l}\left(y\right).  \label{12}
\end{equation}
Hence, if $\vec{H}_{1}(x)=\left(\mathcal{P}_{m,n}\right)^{d}$, then $\vec{H}_{0}(y)=\left(\mathcal{P}_{m,n}\right)^{d}$ for all $y$ in a neighborhood of $x$. Consequently, $\vec{H}_l(y)=\left(\mathcal{P}_{m,n}\right)^{d}$, for all $l \geq 1$, proving \eqref{11}.

For the induction step, fix some $k\geq0$ and assume Lemma \ref{lem1} holds for that $k$.  We must show that
\begin{equation}
\mbox{if }\dim \vec{H}_{2k+3}\left(x\right)\geq\dim\left[\left(\mathcal{P}%
_{m,n}\right)^{d}\right]-\left(k+1\right),\mbox{ then }\vec{H}%
_{l}\left(x\right)=\vec{H}_{2k+3}\left(x\right)\mbox{ for all }l\geq2k+3.
\label{13}
\end{equation}
By the inductive hypothesis, we know that if $\dim \vec{H}%
_{2k+1}\left(x\right)\geq\dim[\left(\mathcal{P}_{m,n}\right)^{d}]-k$, then $%
\vec{H}_{l}\left(x\right)=\vec{H}_{2k+1}\left(x\right)$ for all $l\geq2k+1$;
consequently, (\ref{13}) holds. Thus, to prove (\ref{13}), we may assume
that
\begin{equation*}
\dim \vec{H}_{2k+1}\left(x\right)\leq\dim\left[\left(\mathcal{P}%
_{m,n}\right)^{d}\right]-\left(k+1\right).
\end{equation*}
Thus,\begin{equation}
\dim \vec{H}_{2k+1}\left(x\right)=\dim \vec{H}_{2k+2}\left(x\right)=\dim
\vec{H}_{2k+3}\left(x\right)=\dim\left[\left(\mathcal{P}_{m,n}\right)^{d}%
\right]-\left(k+1\right).  \label{14}
\end{equation}

Note that $\dim \vec{H}_{2k+1}(y)\geq \dim (\mathcal{P}_{m,n})^d-(k+1)$ for all $y$ sufficiently close to $x$ since otherwise \eqref{12} (with $l=2k+1$) would contradict \eqref{14}.

Next, we will show the following:
\begin{equation}
\vec{H}_{2k+2}\left(y\right)=\vec{H}_{2k+1}\left(y\right)\mbox{ for all }%
y \mbox{
sufficiently close to }x.  \label{15}
\end{equation}

Suppose toward a contradiction that (\ref{15}) fails; that is,
\begin{equation}
\mbox{ there exists }y \mbox{ arbitrarily near }x \mbox{ such that } \dim \vec{H}_{2k+2}\left(y\right)\leq\dim \vec{H}_{2k+1}\left(y\right)-1.  \label{16}
\end{equation}
Then, since we are assuming Lemma \ref{lem1} for $k$, we must have
\begin{equation}
\dim \vec{H}_{2k+1}\left(y\right)\leq\dim\left[\left(\mathcal{P}%
_{m,n}\right)^{d}\right]-\left(k+1\right)  \label{17}
\end{equation}
for all $y$ as in \eqref{16}.

This together with (\ref{16}) shows
\begin{equation}
\dim \vec{H}_{2k+2}\left(y\right)\leq\dim\left[\left(\mathcal{P}%
_{m,n}\right)^{d}\right]-k-2\mbox{ for all }y%
\mbox{ arbitrarily close
to }x.  \label{18}
\end{equation}
From (\ref{12}) and \eqref{18}, we get
\begin{equation*}
\dim \vec{H}_{2k+3}\left(x\right)\leq\dim\left[\left(\mathcal{P}%
_{m,n}\right)^{d}\right]-k-2,
\end{equation*}
contradicting (\ref{14}). Therefore, (\ref{15}) cannot fail.

By (\ref{15}), we see easily that $%
\vec{H}_{l}\left(y\right)=\vec{H}_{2k+1}\left(y\right)$ for all $l\geq2k+1$
and all $y\in E$ sufficiently close to $x$. In particular, we have $\vec{H}%
_{l}\left(x\right)=\vec{H}_{2k+3}\left(x\right)$ for all $l\geq2k+3$. This
completes the induction step, and proves Lemma \ref{lem1}.
\end{proof}

According to this lemma, we know that $\vec{\mathcal{H}}_l = \vec{\mathcal{H}%
}_{2\dim [(\mathcal{P}_{m,n})^d]+1}$ for $l \geq 2\dim [(\mathcal{P}%
_{m,n})^d]+1$. Moreover, following the argument for Lemma 2.1 in \cite{F1}, we see that if $\vec{H}(x)=\vec{f}(x)+\vec{I}(x)$ for some $\mathcal{R}%
_{m,n}^{x}$-submodule $\vec{I}(x)$ of $(\mathcal{P}%
_{m,n})^{d}$ and if the Glaeser refinement
$\vec{H}^{\prime }(x)\not=\emptyset$, then $\vec{H}^{\prime }(x)=\vec{f}%
_1(x)+\vec{I}_1(x)$ for some $\vec{f}_1(x)\in \vec{H}^{\prime }(x)$ and some $\mathcal{R}%
_{m,n}^{x}$-submodule $\vec{I}_l(x)$ of $(\mathcal{P}%
_{m,n})^{d}$. Consequently, by
Theorem \ref{Thm3}, we easily see that Problem \ref{problem3-1} is solvable
for $X=C^m(\mathbb{R}^n)$ if and only if the bundle $\vec{\mathcal{H}}%
_{2\dim [(\mathcal{P}_{m,n})^d]+1}$ contains no empty fiber.

\bibliographystyle{plain}
\bibliography{papers}

\end{document}